\documentclass[letterpaper,reqno,12pt]{amsart}
\usepackage[margin=1.2in]{geometry}

\usepackage{amssymb, enumerate}
\usepackage[all]{xy}
\usepackage{hyperref}
\hypersetup{colorlinks=true}

\theoremstyle{plain}
\newtheorem{thm}{Theorem}[section] 
\newtheorem{cor}[thm]{Corollary}
\newtheorem{prop}[thm]{Proposition}
\newtheorem{conj}[thm]{Conjecture}
\newtheorem{lem}[thm]{Lemma}
\newtheorem*{mainthm}{Main Theorem}

\theoremstyle{definition} 
\newtheorem{defn}[thm]{Definition}
\newtheorem{notation}[thm]{Notation}

\newtheorem{eg}[thm]{Example} 

\newtheorem*{Notation}{Notation}

\theoremstyle{remark}
\newtheorem{rem}[thm]{Remark}

\newtheorem*{acknowledgement}{Acknowledgments}

\def\ge{\geqslant}
\def\le{\leqslant}
\def\phi{\varphi}
\def\epsilon{\varepsilon}

\def\Mod{\operatorname{\,mod \,}}

\newcommand{\sO}{\mathcal{O}}

\newcommand{\N}{\mathbb{N}}
\newcommand{\Q}{\mathbb{Q}}

\newcommand{\ba}{\mathfrak{a}}
\newcommand{\m}{\mathfrak{m}}
\newcommand{\n}{\mathfrak{n}}

\newsavebox{\circlebox}
\savebox{\circlebox}{\fontencoding{OMS}\selectfont\Large\char13}
\newlength{\circleboxwdht}

\def\Hom{\operatorname{Hom}}
\def\Spec{\operatorname{Spec}}
\def\Proj{\operatorname{Proj}}

\def\Div{\operatorname{div}}

\title{Finitistic test ideals on numerically $\Q$-Gorenstein varieties}

\author{Shunsuke Takagi}
\address{Graduate School of Mathematical Sciences, University of Tokyo, 3-8-1 Komaba, Meguro-ku, Tokyo 153-8914, Japan}
\email{stakagi@ms.u-tokyo.ac.jp}

\dedicatory{Dedicated to Professor~Craig~Huneke on the occasion of his sixty-fifth birthday.}


\begin{document}

\begin{abstract}
We prove that the finitistic test ideal $\tau_{\rm fg}(R, \Delta, \ba^t)$ coincides with the big test ideal $\tau_{\rm b}(R, \Delta, \ba^t)$ if the pair $(R,\Delta)$ is numerically log $\Q$-Gorenstein. 
\end{abstract}

\maketitle


\section{Introduction}
Let $R$ be an excellent Noetherian domain of characteristic $p>0$. 
The notion of test ideals has its origin in the theory of tight closure, which was introduced about 30 years ago by Hochster and Huneke \cite{HH1}. 
Tight closure is a closure operation performed on ideals in $R$, or more generally, on inclusions of $R$-modules.  
The test ideal defined in \cite{HH1} is generated by test elements, or those elements which can be used to test the tight closure membership for all inclusions of finitely generated $R$-modules. 
This is the etymology of the name ``test ideal". 
This test ideal is nowadays called the finitistic test ideal of $R$ and denoted by $\tau_{\rm fg}(R)$. 

Later, another test ideal, called the big test ideal of $R$ and denoted by $\tau_{\rm b}(R)$, was introduced. 
It is generated by big test elements, or those elements used to test the tight closure membership for all inclusions of (not necessarily finitely generated) $R$-modules. 
The big test ideal fits better into 
the theory of Frobenius splittings and can be characterized without using tight closure. 
Indeed, Schwede \cite{Sc} proved that (if $R$ is $F$-finite) $\tau_{\rm b}(R)$ is the unique smallest nonzero ideal $J \subset R$ such that $\phi(J^{1/p^e}) \subset J$ for all $e \in \N$ and all $\phi \in \Hom_R(R^{1/p^e}, R)$.  
The big test ideal quickly began finding applications in its own right. 
Now this notion is becoming a fundamental tool in birational geometry and singularity theory in positive characteristic.

It is clear from definition that $\tau_{\rm b}(R)$ is contained in $\tau_{\rm fg}(R)$, and they are conjectured to be equal. 
\begin{conj}[\textup{\cite[Conjecture 5.14]{ST}}]\label{main conj}
$\tau_{\rm b}(R)=\tau_{\rm fg}(R)$.
\end{conj}
Conjecture \ref{main conj} is considered to be one of the most important conjectures in tight closure theory. 
When $R$ is normal and $\Q$-Gorenstein, Conjecture \ref{main conj} is known to hold (see \cite{AM}, \cite{BSTZ}, \cite{Ha}, \cite{HY}, \cite{Sm}, \cite{Ta}). 
What if $R$ is normal but not $\Q$-Gorenstein? 
Various results on $\tau_{\rm b}(R)$ in the $\Q$-Gorenstein setting were generalized in \cite{CEMS} to the case where the anti-canonical ring $\bigoplus_{i \ge 0}\sO_X(-iK_X)$ of $X=\Spec R$ (that is, the symbolic Rees algebra of the anti-canonical ideal of $R$) is finitely generated. In particular, they gave an affirmative answer to Conjecture \ref{main conj} in this case. 

In this paper, as another generalization of $\Q$-Gorensteinness, we consider the condition of being numerically $\Q$-Gorenstein. 
The notion of numerically $\Q$-Gorenstein varieties was introduced in \cite{BdFF}, \cite{BdFFU} when the base field is of characteristic zero, and is generalized to arbitrary characteristic in \cite{dFDTT} by making use of regular alterations instead of resolutions of singularities. 
Being numerically $\Q$-Gorenstein is a much weaker condition than being $\Q$-Gorenstein. 
For example, every normal surface is numerically $\Q$-Gorenstein. 
We also note that this condition is somewhat orthogonal to the finite generalization of the anti-canonical ring, because if $R$ satisfies both conditions, then it has to be $\Q$-Gorenstein (see Lemma \ref{Q-Cartier}). 

We give an affirmative answer to Conjecture \ref{main conj} when $R$ is numerically $\Q$-Gorenstein. 
More generally, we can prove the equality for a generalization of test ideals. 
Inspired by the theory of multiplier ideals, Hara-Yoshida \cite{HY} and Takagi \cite{Ta} generalized the big and finitistic test ideals $\tau_{\rm b}(R), \tau_{\rm fg}(R)$ to the ideals $\tau_{\rm b}(R, \Delta, \ba^t), \tau_{\rm fg}(R, \Delta, \ba^t)$ associated to a triple $(R, \Delta, \ba^t)$, where $\Delta$ is an effective $\Q$-Weil divisor on $\Spec R$, $\ba \subset R$ is an ideal, and $t>0$ is a real number. 
Our main result is stated for the ideals $\tau_{\rm b}(R, \Delta, \ba^t)$ and $\tau_{\rm fg}(R, \Delta, \ba^t)$. 

\begin{mainthm}[\textup{Theorem \ref{main thm}, Corollary \ref{two-dimensional}}]
Let $(R, \Delta, \ba^t)$ be a triple where $R$ is an $F$-finite normal domain and set $X=\Spec R$. 
Suppose that one of the following conditions is satisfied: 
\begin{enumerate}
\item[$(\textup{a})$] $R$ is essentially of finite type over a field and $K_X+\Delta$ is numerically $\Q$-Cartier, or
\item[$(\textup{b})$] $R$ is of dimension two. 
\end{enumerate}
Then 
$\tau_{\rm b}(R, \Delta, \ba^t)=\tau_{\rm fg}(R, \Delta, \ba^t)$. 
\end{mainthm}

For the proof, we employ the strategy of MacCrimmon \cite{Mc} and Williams \cite{Wi}. 
The assertion follows from valuative characterizations of tight closure and numerically $\Q$-Cartier divisors (Proposition \ref{valuative criterion} and Lemma \ref{valuation}). 

\begin{small}
\begin{acknowledgement}
 The author is grateful to Sho Ejiri, Osamu Fujino, Kenta Sato, Karl Schwede, Hiromu Tanaka, Kevin Tucker and Sho Yoshikawa for useful conversations. 
He is also indebted to the referee for careful reading of the manuscript and thoughtful  suggestions. 
 He would like to thank Imperial College London for support and hospitality during his visit in Spring 2017, where part of this work was done. 
He was partially supported by JSPS KAKENHI Grant Numbers JP15H03611, JP15KK0152, JP16H02141, and JP17H02831. 
\end{acknowledgement}
\end{small}

\begin{Notation}
Throughout this paper all rings are commutative rings with unity. 
Given a ring $R$, we denote by $R^{\circ}$ the set of elements of $R$ which are not in any minimal prime ideal.
Suppose that $R$ is a normal domain. Given a $\Q$-Weil divisor $\Delta=\sum_i d_i D_i$ on $X=\Spec R$, the round up (resp. the round down) of $\Delta$ is $\lceil \Delta \rceil=\sum_i \lceil d_i \rceil D_i$ (resp. $\lfloor \Delta \rfloor=\sum_i \lfloor d_i \rfloor D_i$) where $\lceil d_i \rceil$ denotes the smallest integer greater than or equal to $d_i$ (resp. $\lfloor d_i \rfloor$ denotes the largest integer less than or equal to $d_i$). 
For an integral Weil divisor $D$ on $X$, we will use the notation $R(D)$ to denote $H^0(X, \sO_X(D))$.  

All schemes are assumed to be Noetherian and separated. 
A variety over a field $k$ means an integral scheme of finite type over $k$.
\end{Notation}

\section{A quick review on test ideals}
In this section, we briefly review the definitions and basic properties of tight closure and test ideals. 

Let $R$ be a normal domain of characteristic $p>0$. 
Given an ideal $I$ of $R$ and an integer $e \in \N$, the $e$-th Frobenius power $I^{[p^e]}$ of $I$ is the ideal of $R$ generated by all $p^e$-th powers of elements of $I$. 

Fix an algebraic closure $L$ of the quotient field $\mathrm{Frac}(R)$ of $R$. 
Given a fractional ideal $J$ of $R$ and an integer $e \in \N$, set $J^{1/p^e}=\{y \in L \; | \; y^{p^e} \in J\}$ and denote an element $y \in J^{1/p^e}$ by $x^{1/p^e}$ where $x=y^{p^e} \in J$. 
We say that $R$ is \textit{$F$-finite} if $R^{1/p}$ is a finitely generated $R$-module. 
By definition, a field $K$ is $F$-finite if and only if $[K:K^p]<\infty$. 
More generally, if $R$ is a complete local ring with $F$-finite residue field or is essentially of finite type over an $F$-finite field, then it is $F$-finite.

\begin{defn}
In this paper, we say that $(R, \Delta, \ba^t)$ is a \textit{triple} in characteristic $p>0$ if $R$ is an excellent normal domain of characteristic $p>0$, $\Delta$ is an effective $\Q$-Weil divisor on $\Spec R$, $\ba$ is a nonzero ideal of $R$ and $t \ge 0$ is a real number. 
We say that the triple $(R, \Delta, \ba^t)$ is \textit{$F$-finite} if so is $R$. 

\end{defn}

\begin{defn}[\textup{\cite[Definition 3.1]{BSTZ}}]\label{tight closure defn}
Let $(R, \Delta, \ba^t)$ be a triple in characteristic $p>0$. 
\begin{enumerate}
\item[(i)] Let $I$ be an ideal in $R$. An element $x \in R$ is said to be in the \textit{$(\Delta, \ba^t)$-tight closure} $I^{*(\Delta, \ba^t)}$ of $I$ if there exists $c \in R^{\circ}$ such that 
\[c \ba^{\lceil t (q-1) \rceil} x^q \subset I^{[q]}R(\lceil (q-1) \Delta \rceil)
\]
for all large $q=p^e$. When $\ba=R$ (resp. $\Delta=0$ and $\ba=R$), we simply denote the ideal $I^{*(\Delta, \ba^t)}$ by $I^{*\Delta}$ (resp. $I^*$). 
\item[(ii)] Let $M$ be an $R$-module. 
An element $z \in M$ is said to be in the \textit{$(\Delta, \ba^t)$-tight closure} $0_M^{*(\Delta, \ba^t)}$ of the zero submodule in $M$ if there exists $c \in R^{\circ}$ such that 
\[(c \ba^{\lceil t (q-1) \rceil})^{1/q} \otimes z=0 \textup{ in } R(\lceil (q-1) \Delta \rceil)^{1/q} \otimes_R M
\]
for all large $q=p^e$. 
We say that $z$ is in the \textit{finitistic tight closure} $0_M^{*(\Delta, \ba^t){\rm fg}}$ of the zero submodule in $M$ if there exists a finitely generated $R$-submodule $N \subset M$ such that $z \in 0_{N}^{*(\Delta, \ba^t)}$. In other words, 
\[0_M^{*(\Delta, \ba^t){\rm fg}}=\bigcup_{N \subset M} 0_{N}^{*(\Delta, \ba^t)}\] where $N$ runs through all finitely generated $R$-submodules of $M$.  
\item[(iii)] We say that an element $c \in R^{\circ}$ is a \textit{big sharp test element} for $(R, \Delta, \ba^t)$ if for all $R$-modules $M$ and all $z \in 0_M^{*(\Delta, \ba^t)}$, we have 
\[(c \ba^{\lceil t (q-1) \rceil})^{1/q} \otimes z=0 \textup{ in } R(\lceil (q-1) \Delta \rceil)^{1/q} \otimes_R M
\]
for every $q=p^e$. 
When $\Delta=0$ and $\ba=R$, we refer to such elements as big test elements for $R$. 
\end{enumerate}
\end{defn}

\begin{rem}
Definition \ref{tight closure defn} (i), (ii) do not change even if we replace $\ba^{\lceil t (q-1) \rceil}$ (resp. $R(\lceil (q-1) \Delta \rceil)$ by $\ba^{\lceil tq \rceil}$ (resp. $R(\lceil q \Delta \rceil)$). 
However, Definition \ref{tight closure defn} (iii) does change if we make such a replacement. 
\end{rem}

Big sharp test elements always exist if the ring is $F$-finite. 
\begin{lem}[\textup{\cite[Lemma 2.17]{Sc}}]\label{test element exists}
Let $(R, \Delta, \ba^t)$ be an $F$-finite triple. 
If we choose an element $c \in R^{\circ}$ such that the localization $R_c$ is regular, $\Div(c) \ge \Delta$ and $\ba R_c=R_c$, then some power $c^n$ of $c$ is a big sharp test element for $(R, \Delta, \ba^t)$. 
\end{lem}

Now we introduce test ideals. There are two kinds of test ideals, finitistic test ideals and big test ideals. 
\begin{defn}[\textup{\cite[Definition-Propositions 3.2 and 3.3]{BSTZ}, cf.~\cite{HH1}}]
Let $(R, \Delta, \ba^t)$ be a triple in characteristic $p>0$ and $E=\bigoplus_{\m}E_R(R/\m)$ be the direct sum, taken over all maximal ideals $\m$ of $R$, of the injective hulls $E_R(R/\m)$ of the residue fields $R/\m$. 
\begin{enumerate}
\item[(i)] The \textit{finitistic test ideal} $\tau_{\rm fg}(R, \Delta, \ba^t)$ associated to $(R, \Delta, \ba^t)$ is defined by 
\[\tau_{\rm fg}(R, \Delta, \ba^t)=\mathrm{Ann}_R(0^{*(\Delta, \ba^t){\rm fg}}_E)=\bigcap_M \mathrm{Ann}_R(0^{*(\Delta, \ba^t)}_M),\] where $M$ runs through all finitely generated $R$-submodules of $E$. 
When $\ba=R$ (resp. $\Delta=0$ and $\ba=R$), we denote it by $\tau_{\rm fg}(R, \Delta)$ (resp. $\tau_{\rm fg}(R)$). 
\item[(ii)] The \textit{big test ideal} $\tau_{\rm b}(R, \Delta, \ba^t)$ associated to $(R, \Delta, \ba^t)$ is defined by 
\[\tau_{\rm b}(R, \Delta, \ba^t)=\mathrm{Ann}_R(0^{*(\Delta, \ba^t)}_E).\] 
When $\ba=R$ (resp. $\Delta=0$ and $\ba=R$), we denote it by $\tau_{\rm b}(R, \Delta)$ (resp. $\tau_{\rm b}(R)$). 
\item[(iii)]
Suppose in addition that $R$ is $F$-finite. 
The ring $R$ is said to be \textit{strongly $F$-regular} if $\tau_{\rm b}(R)=R$. 
\end{enumerate}
\end{defn}

\begin{rem}
The finitistic test ideal was introduced by Hochster-Huneke \cite{HH1} about 30 years ago, but 
the notations used for test ideals vary in the literature. 
The notation $\tau(R)$ was used for the finitistic test ideal $\tau_{\rm fg}(R)$ originally, but nowadays is more often used for the big test ideal $\tau_{\rm b}(R)$, which is sometimes denoted by $\widetilde{\tau}(R)$ in the literature. 
\end{rem}

\begin{lem}[\textup{cf.~\cite[Proposition 8.23 (e)]{HH1}}]\label{finite test completion}
Let $(R, \Delta, \ba^t)$ be a triple where $(R, \m)$ is an $F$-finite local ring of characteristic $p>0$. 
Let $\widehat{R}$ denote the $\m$-adic completion of $R$ and $\widehat{\Delta}$ denote the pullback of $\Delta$ to $\Spec \widehat{R}$. 
Then 
\[\tau_{\rm fg}(R, \Delta, \ba^t)=\tau_{\rm fg}\bigl(\widehat{R}, \widehat{\Delta}, (\ba \widehat{R})^t \bigr) \cap R.\]
\end{lem}
\begin{proof}
Since excellent reduced local rings are approximately Gorenstein by \cite{Ho}, there exists a sequence $\{I_t\}$ of $\m$-primary irreducible ideals in $R$ cofinal with the powers of $\m$.  
Then by the same argument as the proof of \cite[Proposition 8.23 (f)]{HH1}, one has 
\begin{align*}
\tau_{\rm fg}(R, \Delta, \ba^t)&=\bigcap_t (I_t:_R I_t^{*(\Delta, \ba^t)}), \\
\tau_{\rm fg}\bigl(\widehat{R}, \widehat{\Delta}, (\ba \widehat{R})^t \bigr)&=\bigcap_t \bigl(I_t \widehat{R}:_{\widehat{R}}(I_t \widehat{R})^{*(\widehat{\Delta}, (\ba \widehat{R})^t)}\bigr). 
\end{align*}
Thus, it suffices to prove that $I_t^{*(\Delta, \ba^t)} \widehat{R}=(I_t \widehat{R})^{*(\widehat{\Delta}, (\ba \widehat{R})^t)}$. 
Since $(I_t \widehat{R})^{*(\widehat{\Delta}, (\ba \widehat{R})^t)}$ is an $\m \widehat{R}$-primary ideal or the unit ideal of $\widehat{R}$, there exists an ideal $J$ in $R$ such that $J \widehat{R}=(I_t \widehat{R})^{*(\widehat{\Delta}, (\ba \widehat{R})^t)}$. 
It is enough to show that $I_t^{*(\Delta, \ba^t)}=J$, that is, an element $x \in R$ lies in $I_t^{*(\Delta, \ba^t)}$ if and only if $x$ lies in $(I_t \widehat{R})^{*(\widehat{\Delta}, (\ba \widehat{R})^t)}$. 
Let $c \in R^{\circ}$ be a big sharp test element both for $(R, \Delta, \ba^t)$ and for $(\widehat{R}, \widehat{\Delta}, (\ba \widehat{R})^t)$ (Lemma \ref{test element exists} produces such an element).  
By definition, $x \in I_t^{*(\Delta, \ba^t)}$ if and only if $c x^q \ba^{\lceil t (q-1) \rceil}  \subset I^{[q]}R(\lceil (q-1) \Delta \rceil)$ for all $q=p^e$. 
However, since $R \hookrightarrow \widehat{R}$ is faithfully flat, this is equivalent to saying that $c x^q (\ba \widehat{R})^{\lceil t (q-1) \rceil} \subset I^{[q]} \widehat{R}(\lceil (q-1) \widehat{\Delta} \rceil)$ for all $q=p^e$, which implies that $x \in (I_t \widehat{R})^{*(\widehat{\Delta}, (\ba \widehat{R})^t)}$. 
\end{proof}

\begin{rem}\label{finite test rem}
(1) When $\Delta=0$ and $\ba=R$, it follows from \cite[Proposition 8.23 (e)]{HH1} that Lemma \ref{finite test completion} holds for arbitrary excellent reduced local rings of characteristic $p>0$, not necessarily $F$-finite. 

(2) (\cite[Remark 3.6]{BSTZ}, \cite[Proposition 3.2]{HT}) Suppose that we are in the setting of Lemma \ref{finite test completion}. In the case of big test ideals, a similar but stronger statement holds: 
\[\tau_{\rm b}(R, \Delta, \ba^t)\widehat{R}=\tau_{\rm b}\bigl(\widehat{R}, \widehat{\Delta}, (\ba \widehat{R})^t \bigr).\]
\end{rem}

\section{Numerically $\Q$-Cartier divisors}
In this section, we recall the notion of numerically $\Q$-Cartier divisors, introduced in \cite{BdFF} and \cite{BdFFU}, which is a natural generalization of $\Q$-Cartier divisors. 
 
First we note that the negativity lemma holds for normal varieties over any field. 

\begin{lem}[Negativity lemma]\label{negativity lemma}
Let $h:Z \to Y$ be a birational morphism between normal varieties over a field $k$. 
Suppose that $-B$ is an $h$-nef $\Q$-Cartier $\Q$-Weil divisor on $Z$. 
Then $B$ is effective if and only if $h_*B$ is. 
\end{lem}
\begin{proof} 
After extending the base field and taking the normalization, we can reduce to the case where the base field is algebraically closed. 
The assertion now follows from \cite[2.3]{Bir}. 
\end{proof}

Using regular alterations, we define numerically $\Q$-Cartier divisors. 
Here an \textit{alteration} of a variety $X$ over a field is a generically finite proper surjective morphism $\pi:Y \to X$ and a \textit{regular alteration} of $X$ is an alteration $\rho:Z \to X$ from a regular variety $Z$, which always exists by \cite{dJ}. 

\begin{defn}[\textup{\cite[Definition 15]{dFDTT}, cf.~\cite[Definition 5.2]{BdFFU}}]\label{num def}
Let $X$ be a normal variety over a field. 
\begin{enumerate}
\item[(i)] Suppose that $D$ is a $\Q$-Weil divisor on $X$ and $x \in X$ is a (not necessarily closed) point of $X$.  
We say that $D$ is \textit{numerically $\Q$-Cartier} at $x$ if there exist a open neighborhood $U$ of $x$, a regular alteration $\pi:V \to U$, and a $\pi$-numerically trivial $\Q$-divisor $D_V$ on $V$ such that $D|_U=\pi_*D_V$.
We also say that $D$ is numerically $\Q$-Cartier if $D$ is numerically $\Q$-Cartier at all points of $X$. 
\item[(ii)] 
$X$ is said to be \textit{numerically $\Q$-Gorenstein} if $K_X$ is numerically $\Q$-Cartier.  
$X$ is said to be \textit{numerically $\Q$-factorial} if every $\Q$-Weil divisor on it is numerically $\Q$-Cartier. 
\end{enumerate}
\end{defn}

\begin{rem}\label{num Cartier rem}
Definition \ref{num def} (i) is equivalent to saying that there exists a open neighborhood $U$ of $x$, a regular alteration $\pi:V \to U$, and a $\Q$-divisor $D'_V$ on $V$ such that $D'_V$ is $f$-numerically trivial and $f_*D'_V=g^*D|_U$, where $\pi:V \xrightarrow{f} W \xrightarrow{g} U$ is the stein factorization of $\pi$. 
The uniqueness of $D'_V$ can be verified by applying Lemma \ref{negativity lemma} to $f$. 
Therefore, we denote $D'_V$ by $\pi_{\rm num}^*D|_U$ and call it the numerical pullback of $D|_U$ to $V$.
It follows from essentially the same argument as in \cite[Proposition 5.3]{BdFFU} that the numerical pullback of $D|_U$ can be defined for an arbitrary regular alteration of $U$.  
\end{rem}

\begin{eg}\label{eg num Cartier}
(1) $\Q$-Cartier $\Q$-Weil divisors are clearly numerically $\Q$-Cartier. 

(2) Every two-dimensional excellent normal integral scheme is numerically $\Q$-factorial, because Mumford's pullback for $\Q$-Weil divisors exists by the Hodge index theorem \cite[Theorem 10.1]{Ko}. 
\end{eg}

The following lemma tells us that the condition of $D$ being numerically $\Q$-Cartier is somewhat orthogonal to the finite generation of the section ring $\bigoplus_{i \ge 0}\sO_X(\lfloor iD \rfloor)$. 

\begin{lem}\label{Q-Cartier}
Let $x \in X$ be a point of a normal variety $X$ over a field and $D$ be a $\Q$-Weil divisor on $X$. 
Suppose that $\mathcal{R}=\bigoplus_{i \ge 0}\sO_X(\lfloor iD \rfloor)$ is Noetherian. 
Then $D$ is numerically $\Q$-Cartier at $x$ if and only if it is $\Q$-Cartier at $x$. 
\end{lem}
\begin{proof}
Suppose that $D$ is numerically $\Q$-Cartier at $x$. 
After shrinking $X$ if necessary, we can define the numerical pullback of $D$ to an arbitrary regular alteration of $X$. 
Let $\rho:Z=\Proj \mathcal{R} \to X$ be the structure morphism. Then $\rho$ is a small morphism and $\rho^{-1}_*D$ is a $\rho$-ample $\Q$-Cartier $\Q$-Weil divisor on $Z$. 
We will show that $\rho$ is an isomorphism. 

Take a regular alteration $\mu: Y \to Z$, and then the composite morphism $\pi:Y \xrightarrow{\mu} Z \xrightarrow{\rho} X$ is a regular alteration of $X$. 
Since $\mu_* \pi^*_{\rm num}D=(\deg \mu) \rho^{-1}_*D$, it follows from the uniqueness of the numerical pullback of $\rho^{-1}_*D$ (see Remark \ref{num Cartier rem}) that $\pi^*_{\rm num}D= \mu^*\rho^{-1}_*D$. 
Suppose to the contrary that $\rho$ is not an isomorphism, which implies that there exists a $\rho$-exceptional curve $C$ on $Z$. 
Then 
\[(\pi^*_{\rm num}D, \mu^{-1}_*C)=\left(\mu^*\rho^{-1}_*D,  \mu^{-1}_*C\right)=(\rho^{-1}_*D,  C)>0,\]
because $\rho^{-1}_*D$ is $\rho$-ample. 
This, however, contradicts the fact that $\pi^*_{\rm num}D$ is numerically $\pi$-trivial. 
\end{proof}

\begin{eg}[\textup{cf.~\cite[Example 5.8]{BdFFU}}]
Let $x \in X$ be a closed point of a three-dimensional strongly $F$-regular (affine) variety over an algebraically closed field of characteristic $p > 5$ and let $D$ be a $\Q$-Weil divisor on $X$.  
It follows from \cite[Theorem 4.3]{SS} and \cite[Theorem 3.3]{HW} that there exists an effective $\Q$-Weil divisor $\Delta$ on $X$ such that $K_X+\Delta$ is $\Q$-Cartier and $(X, \Delta)$ is klt.
Then $\bigoplus_{i \ge 0}\sO_X(\lfloor iD \rfloor)$ is Noetherian by \cite[Lemma 2.27]{CEMS}. 
Applying Lemma \ref{Q-Cartier}, we see that $D$ is numerically $\Q$-Cartier at $x$ if and only if it is $\Q$-Cartier at $x$. 
\end{eg}

We will use the following fact in the proof of the main theorem. 

\begin{lem}[\textup{\cite[Proposition 5.9]{BdFFU}}]\label{valuation}
Let $z \in Z$ be a point of a normal variety $Z$ over a field and let $B$ be a Weil divisor on $Z$ which is numerically $\Q$-Cartier at $z$. 
Then for every divisorial valuation $\nu$ on $Z$ centered at $z$, we have 
\[\lim_{m \to \infty}\frac{\nu\bigl(\sO_Z(mB)\sO_Z(-mB)\bigr)}{m}=0, \]
where $\nu(\sO_Z(mB)\sO_Z(-mB))=\min\{\nu(f) \, | \, f \in \sO_Z(mB)_z\sO_Z(-mB)_z \subset \sO_{Z, z}\}$. 
\end{lem}

\begin{proof}
After shrinking $Z$ if necessary, we may assume that $B$ is a numerically $\Q$-Cartier Weil divisor on $Z$. 
The assertion then follows from essentially the same argument as in \cite[Proposition 5.9]{BdFFU}. 
\end{proof}

\section{Main Theorem}
In this section, we will prove the main theorem. 
The following valuative characterization of tight closure plays an important role in the proof. 
\begin{prop}[\textup{cf.~\cite[Theorem 3.1]{HH}}]\label{valuative criterion}
Let $(R, \Delta, \ba^t)$ be a triple where $(R, \m)$ is a $F$-finite complete local ring of characteristic $p>0$. 
Suppose that $S$ is a normal domain which is a module-finite extension of $R$ and $\rho^*\Delta$ is an integral divisor on $Z=\Spec S$, where  $\rho: Z \to X$ is the morphism induced by the inclusion $R \hookrightarrow S$. 
Let $u$ be an element of $R$ and $I$ be an ideal of $R$, and fix a $\Q$-valued valuation $\nu$ on $S$ which is nonnegative on $R$ and positive on $\m$.  
Then $u \in I^{*(\Delta, \ba^t)}$ if and only if there exists a sequence of nonzero elements $c_e \in S$ satisfying the following three conditions: 
\begin{enumerate}
\item[\textup{(i)}] $\{\nu(c_{e})/p^{e}\}$ is a monotonically decreasing sequence, 
\item[\textup{(ii)}] $\lim_{e \to \infty} \nu(c_e)/p^e=0$, 
\item[\textup{(iii)}] $c_e u^{p^e} \ba^{\lceil tp^e \rceil} S \subset I^{[p^e]}S(p^e \rho^*\Delta)$ for all $e$.
\end{enumerate}
\end{prop}

\begin{proof}
Suppose that there exists a sequence of nonzero elements $c_e \in S$ satisfying the conditions (i)-(iii). 
First note that $\nu$ can be extended to a $\Q$-valued valuation on the absolute integral closure $R^+$ of $R$ (we use the same letter $\nu$ to denote this valuation on $R^+$). 
It follows from \cite[Theorem 3.3]{HH} that there exist a real number $\lambda>0$ and an integer $r>0$ such that for every element $s \in R^+$ with $\nu(s)<\lambda$, there is an $R$-linear map $\phi:R^+ \to R$ such that $\phi(s) \notin \m^r$. 
Fix any power $q=p^{e}$ of $p$, and choose $q'=p^{e'}$ so large that $\nu(c_{ee'}^{1/q'})=\nu(c_{ee'})/q' \le \nu(c_{e'})q/q'<\lambda$, where the second inequality follows from (i) and the third one does from (ii). 
We have $c_{ee'} u^{qq'} \ba^{\lceil tqq' \rceil}S \subset I^{[qq']} S(qq' \rho^*\Delta)$ by (iii) and taking the $q'$-th root on both sides yields that 
\begin{equation}
c_{ee'}^{1/q'} u^{q} \ba^{\lceil tq \rceil}S^{1/q'} \subset I^{[q]} S(qq' \rho^*\Delta)^{1/q'}.\tag{$\clubsuit$}
\end{equation}
By the choice of $\lambda$ and $r$, there exists an $R$-linear map $\widetilde{\phi}:S^{1/q'} \to R$ such that $\widetilde{\phi}(c_{ee'}^{1/q'}) \notin \m^r$, which can be extended to an $R$-linear map $S(qq' \rho^* \Delta )^{1/q'} \to R(\lceil q \Delta \rceil)$ (we use the same letter $\widetilde{\phi}$ to denote this map). 
Setting $\widetilde{c_q}=\widetilde{\phi}(c_{ee'}^{1/q'}) \in R \setminus \m^r$ and applying $\widetilde{\phi}$ to $(\clubsuit)$, we have 
$\widetilde{c_q}u^{q} \ba^{\lceil tq \rceil}  \subset I^{[q]} R(\lceil q \Delta \rceil)$ 
for all $q=p^e$. 

Set $J_q=((I^{[q]})^{*\lceil q \Delta \rceil}:_R u^{q} \ba^{\lceil tq \rceil})$ for each $q=p^e$. 
Note that $\widetilde{c_q} \in J_q$, because $I^{[q]}R(\lceil q \Delta \rceil) \cap R$ is contained in $(I^{[q]})^{*\lceil q \Delta \rceil}$. 
Then $\{J_q\}$ is a decreasing sequence of ideals of $R$. 
Indeed, if $x u^{pq } \ba^{\lceil t pq \rceil} \subset (I^{[pq]})^{*\lceil pq \Delta \rceil}$ with $x \in R$, 
then there exists a nonzero element $b \in R$ such that 
\[b (xu^{q})^{pQ}(\ba^{\lceil tq \rceil})^{[pQ]}  \subset b (xu^{pq})^Q(\ba^{\lceil tpq \rceil})^{[Q]} \subset I^{[pq Q]} R(Q \lceil pq \Delta \rceil) \subset I^{[pq Q]} R(pQ \lceil q \Delta \rceil)\] for all large powers $Q$ of $p$, 
which implies that $ x u^{q}\ba^{\lceil tq \rceil} \subset (I^{[q]})^{*\lceil q \Delta \rceil}$. 
Since the sequence $\{J_q\}$ is decreasing, if it had intersection (0), then Chevalley's theorem would force $J_q \subset \m^r$ for sufficiently large $q$. 
This contradicts the fact that $\widetilde{c_q} \in J_q \setminus \m^r$, so we can choose a nonzero element $d \in \bigcap_{q}J_q$, that is, $d u^q \ba^{\lceil tq \rceil}  \subset (I^{[q]})^{*\lceil q \Delta \rceil}$ for all $q=p^e$. 
Let $c \in R^{\circ}$ be a $\Delta$-test element in the sense of \cite{Ta} (see \cite[Definition 2.4]{Ta} for its definition and \cite[Lemma 2.5]{Ta} for the existence of such an element). 
It follows from the definition of $\Delta$-test elements that $c (I^{[q]})^{*\lceil q \Delta \rceil} \subset I^{[q]}R(\lceil q \Delta \rceil)$, so that $cd  u^{q} \ba^{\lceil tq \rceil} \subset I^{[q]}R(\lceil q \Delta \rceil)$ for all $q=p^e$. Therefore, $u \in I^{*(\Delta, \ba^t)}$ as required. 
\end{proof}

In order to prove the main theorem, we employ the strategy of MacCrimmon \cite{Mc} and Williams \cite{Wi}.  The following notation is due to \cite{Wi}. 

\begin{notation}
Let $R$ be a Noetherian local ring and $\mathbf{x}=x_1, \dots, x_n$ be a system of parameters for $R$. 
For an integer $t \ge 1$, we write $\mathbf{x}^t=x_1^t, \dots, x_n^t$ and denote the ideal $(x_1^t, \dots, x_n^t)$ by $(\mathbf{x}^t)$.  
For every $R$-module $M$, we also denote by $\mathcal{K}(\mathbf{x}, t, M)$ the kernel of the map $M/(\mathbf{x})M \xrightarrow{\times (x_1 \cdots x_n)^{t-1}}M/(\mathbf{x}^t)M$ induced by the multiplication by the element $(x_1 \cdots x_n)^{t-1}$, and we write 
\[
\mathcal{K}(\mathbf{x}, \infty, M):=\bigcup_{t \ge 1} \mathcal{K}(\mathbf{x}, t, M).
\]
\end{notation}

\begin{lem}[\textup{cf.~\cite{Wi}}]\label{key prop}
Let $(R, \Delta, \ba^t)$ be a triple where $(R, \m)$ is an $n$-dimensional $F$-finite complete local ring of characteristic $p>0$. 
Suppose that $(S,\n)$ is a normal local ring which is a module-finite extension of $R$ and $\rho^*\Delta$ is an integral divisor on $Z=\Spec S$, where $\rho: Z \to X$ is the morphism induced by the inclusion $R \hookrightarrow S$. 
Let $\mathbf{x}=x_1, \dots, x_n$ be a system of parameters for $R$ and $J$ be a divisorial ideal of $R$, and fix a $\Q$-valued valuation $\nu$ on $S$ which is nonnegative on $R$ and  positive on $\m$.  
Then $0^{*(\Delta, \ba^t)}_{H^n_{\m}(J)}=0^{*(\Delta, \ba^t){\rm fg}}_{H^n_{\m}(J)}$ if there exists a sequence of nonzero elements $c_e \in S$ satisfying the following three conditions: 
\begin{enumerate}
\item [\textup{(i)}] $\{\nu(c_{e})/p^{e}\}$ is a monotonically decreasing sequence, 
\item[\textup{(ii)}] $\lim_{e \to \infty} \nu(c_e)/p^e=0$, 
\item[\textup{(iii)}] there exists an integer $t_0 \ge 2 $ such that for all $s \ge 1$ and all $q=p^e$, 
\[
c_e \mathcal{K}\bigl(\mathbf{x}^{qs}, \infty, J^{[q]}S(q \rho^*\Delta)\bigr) \subset \mathcal{K}\bigl(\mathbf{x}^{qs}, t_0, J^{[q]}S(q \rho^*\Delta)\bigr).
\]
\end{enumerate}
\end{lem}

\begin{proof}
Let $\xi \in 0^{*(\Delta, \ba^t)}_{H^n_{\m}(J)}$, that is, there exists a nonzero element $d \in R$ such that $(d \ba^{\lceil tq \rceil})^{1/q} \otimes \xi=0$ in $R(\lceil q\Delta \rceil)^{1/q} \otimes_R H^n_{\m}(J)$ for all $q=p^e$. 
We identify $H^n_{\m}(J)$ with $\varinjlim J/(\mathbf{x}^{t})J$ and represent $\xi$ by the image of $z \Mod (\mathbf{x}^s)J \in J/(\mathbf{x}^{s})J$ for some $z \in J$ and $s \ge 1$. 
Since the natural map 
\[R(\lceil q \Delta \rceil)^{1/q} \otimes_R H^n_{\m}(J) \to \varinjlim \bigl(J^{[q]}R(\lceil q \Delta \rceil)\bigr)/\bigl((\mathbf{x}^{qt})J^{[q]}R(\lceil q \Delta \rceil)\bigr)\]
sends $1^{1/q} \otimes \xi$ to the image of $z^q \Mod (\mathbf{x}^{qs})J^{[q]}R(\lceil q \Delta \rceil)$, 
\[\textup{the image of } d z^q \ba^{\lceil tq \rceil}  \Mod (\mathbf{x}^{qs})J^{[q]}R(\lceil q \Delta \rceil)=0
\]
in $\varinjlim (J^{[q]}R(\lceil q \Delta \rceil))/((\mathbf{x}^{qt})J^{[q]}R(\lceil q \Delta \rceil))$  
for all $q=p^e$. 
Therefore, 
\[dd' z^q \ba^{\lceil tq \rceil}S \subset \mathcal{K}\bigl(\mathbf{x}^{qs}, \infty, J^{[q]}S(q\rho^*\Delta)\bigr)\] 
for a nonzero element $d' \in R(-\lceil \Delta \rceil)$, which implies by the condition (iii) that 
\[c_e dd' z^q \ba^{\lceil tq \rceil} S \subset \mathcal{K}\bigl(\mathbf{x}^{qs}, t_0, J^{[q]}S(q\rho^*\Delta)\bigr),\] that is, 
$c_e d d' z^q  (x_1 \cdots x_n)^{qs(t_0-1)} \ba^{\lceil tq \rceil} S \subset ((\mathbf{x}^{st_0})J)^{[q]}S(q\rho^*\Delta)$ for all $q=p^e$. 
Applying Proposition \ref{valuative criterion}, one has $z(x_1 \cdots x_n)^{s(t_0-1)} \in ((\mathbf{x}^{st_0}) J)^{*(\Delta, \ba^t)}$. 
This means that $\xi$, which is represented by the image of $z(x_1 \cdots x_n)^{s(t_0-1)} \Mod  (\mathbf{x}^{st_0})J \in J/(\mathbf{x}^{st_0})J$, belongs to $0^{*(\Delta, \ba^t)}_{N}$, where $N$ is the submodule of $H^n_{\m}(J)$ generated by the image of $J/ (\mathbf{x}^{st_0})J$. 
Thus, $\xi \in 0^{*(\Delta, \ba^t){\rm fg}}_{H^n_{\m}(J)}$. 
\end{proof}

Suppose that $R$ is a normal domain essentially of finite type over a field $k$ and $D$ is a $\Q$-Weil divisor on $X=\Spec R$. 
We say that $D$ is \textit{numerically $\Q$-Cartier} if there exist a normal affine variety $X'=\Spec R'$ over $k$ and a $\Q$-Weil divisor $D'$ on $X'$ such that $R$ is a localization of $R'$, $D$ is the pullback of $D'$ to $X$, and $D'$ is numerically $\Q$-Cartier at all $x \in X \subset X'$. 

We are now ready to prove our main theorem, which is a generalization of \cite[Proposition 3.7]{BSTZ} (cf.~\cite[Definition-Theorem 6.5]{HY}, \cite[Theorem 2.8 (2)]{Ta}). 
\begin{thm}\label{main thm}
Let $(R, \Delta, \ba^t)$ be a triple where $R$ is a normal domain essentially of finite type over an $F$-finite field $k$ of characteristic $p>0$. 
If $K_X+\Delta$ is numerically $\Q$-Cartier where $X=\Spec R$, then 
$\tau_{\rm b}(R, \Delta, \ba^t)=\tau_{\rm fg}(R, \Delta, \ba^t)$. 
\end{thm}
\begin{proof}
We may assume that $(R,\m)$ is a local ring by essentially the same argument as the proof of \cite[Proposition 3.7]{BSTZ}. 
Choose a point $x$ of a normal affine variety $X'$ over $k$ and a $\Q$-Weil divisor $\Delta'$ on $X'$ such that $R \cong \sO_{X', x}$, $\Delta \cong \Delta'_x$ and $K_{X'}+\Delta'$ is numerically $\Q$-Cartier at $x$. 
Take an effective Weil divisor  $D'$ on $X'$ such that $D' \sim -K_{X'}$. 
Let $\pi':Y' \to X'$ be a regular alteration and $Y' \xrightarrow{\mu'} Z' \xrightarrow{\rho'} X'$ denote its Stein factorization. 
After possibly passing to a further alteration, we may assume that ${\rho'}^*(\Delta'-D')$ is an integral divisor on $Z'$. 
Fix a point $z \in Z'$ lying over $x \in X'$ and a divisorial valuation $\nu$ on $Z'$ centered at $z$. 
By Lemma \ref{valuation}, 
\[\lim_{m \to \infty}\frac{\nu \Bigl(\sO_{Z'} \bigl(m{\rho'}^*\left(\Delta'-D'\right)\bigr)\sO_{Z'}\bigl(-m{\rho'}^*\left(\Delta'-D'\right)\bigr)\Bigr)}{m} =0.\] 
Let $\pi:Y \xrightarrow{\mu} Z \xrightarrow{\rho} X=\Spec R$ be the flat base change of $\pi'=\mu' \circ \rho'$ via $X \to X'$ and $D$ be the pullback of $D'$ to $X$. 
Then $D$ is an effective Weil divisor on $X$ such that $D \sim -K_X$ and $\rho^*(\Delta-D)$ is an integral divisor on $Z$. 
Set $S=\rho_*\sO_Z$, which is a normal domain and a module-finite extension of $R$. 

It follows from Lemma \ref{finite test completion} and Remark \ref{finite test rem} (2) that 
\begin{align*}
\tau_{\rm fg}(R, \Delta, \ba^t)&=\tau_{\rm fg}\bigl(\widehat{R}, \widehat{\Delta}, (\ba \widehat{R})^t \bigr) \cap R,\\
\tau_{\rm b}(R, \Delta, \ba^t)&=\tau_{\rm b}\bigl(\widehat{R}, \widehat{\Delta}, (\ba \widehat{R})^t \bigr) \cap R,
\end{align*} 
where $\widehat{R}$ is the $\m$-adic completion of $R$ and $\widehat{\Delta}$ is the pullback of $\Delta$ to $\Spec \widehat{R}$.  
On the other hand, 
the $\m$-adic completion $S \otimes_R \widehat{R}$ of $S$ is a finite product $S_1 \times \cdots \times S_r$ of complete normal local rings $(S_i, \n_i)$ which are module-finite extensions of $\widehat{R}$. 
Note that after reindexing if necessary, we can assume that $(S_1, \n_1)$ is isomorphic to the maximal-ideal-adic completion of $\sO_{Z',z}$. 
Therefore, by \cite[Proposition 9.3.5]{HS}, the divisorial valuation $\nu$ can be extended to a divisorial valuation $\widehat{\nu}$ on $\Spec S_1$ centered at $\n_1$ such that 
\[
\lim_{m \to \infty}\frac{\widehat{\nu}\Bigl(S_1\bigl(m{\rho_1}^*(\widehat{\Delta}-\widehat{D})\bigr)S_1\bigl(-m{\rho_1}^*(\widehat{\Delta}-\widehat{D})\bigr)\Bigr)}{m} =0, 
\]
where $\rho_1:\Spec S_1 \to \Spec \widehat{R}$ is the finite morphism induced by $\rho$ and $\widehat{D}$ is the pullback of $D$ to $\Spec \widehat{R}$. 
Thus, by passing to completion, we may assume that $(R, \m)$ and $(S, \n)$ are complete local rings of characteristic $p>0$ and of dimension $n \ge 2$, $\nu$ is a divisorial valuation on $\Spec S$ centered at $\n$, and 
\begin{equation}
\lim_{m \to \infty}\frac{\nu\Bigl(S\bigl(m{\rho}^*(\Delta-D)\bigr)S\bigl(-m{\rho}^*(\Delta-D)\bigr)\Bigr)}{m} =0. \tag{$\spadesuit$}
\end{equation}
Set $J=R(-D) \subset R$, and we will show that $0^{*(\Delta, \ba^t)}_{H^n_{\m}(J)}=0^{*(\Delta, \ba^t){\rm fg}}_{H^n_{\m}(J)}$. 

Let $x_1 \in J$ be an arbitrary nonzero element.
We can choose an element $x_2 \in R(D)R(-D)$ such that $x_2$ is $R/(x_1)$-regular, because $R(D)R(-D) \subset R$ is the unit ideal or has height $\ge 2$. 
We extend $x_1, x_2$ to a system of parameters $x_1, x_2, \dots, x_n$ for $R$. 
For each $e \ge 1$, set $M_e=S(p^e \rho^*(\Delta-D))$. 
Then 
\begin{align*}
& J^{[q]} S(q\rho^*\Delta)  \subseteq M_e \subseteq S(q\rho^*\Delta), \\
& x_2^{q} M_e  \subseteq J^{[q]} R(qD) M_e  \subseteq J^{[q]} S(q\rho^*\Delta)
\end{align*}
for all $q=p^e$. 
Since $x_1^q \in J^{[q]}$, this implies that $(x_1 \cdots x_{i-1}x_{i+1} \cdots x_n)^{qs}M_e$ is contained in $J^{[q]}S(q\rho^*\Delta)$ for all integers $s \ge 1$ and $1 \le i \le n$. 
Therefore, the map 
\[
H^{n-1}\bigl(\mathbf{x}^{qrs}; M_e/J^{[q]}S(q\rho^*\Delta)\bigr) \to H^{n-1}\bigl(\mathbf{x}^{q(r+1)s}; M_e/J^{[q]}S(q\rho^*\Delta)\bigr)
\]
between Koszul cohomology modules is zero for all $r,s \ge 1$ and all $q=p^e$. 
On the other hand, fix any power $q=p^e$ of $p$ and suppose that $z \Mod (\mathbf{x}^{qs})M_e \in \mathcal{K}(\mathbf{x}^{qs}, \infty, M_e)$ with $z \in M_e$, 
that is, there exists an integer $r \ge 1$ such that $(x_1 \cdots x_n)^{q(r-1)s}z \in (\mathbf{x}^{qrs})M_e$.  
Since 
\[(x_1 \cdots x_n)^{q(r-1)s}  z S\bigl(q\rho^*(D-\Delta)\bigr)  \subseteq  (\mathbf{x}^{qrs}) S\bigl(q\rho^*(D-\Delta)\bigr)  M_e \subseteq (\mathbf{x}^{qrs})S,\]
it follows from the colon-capturing property of tight closure \cite[Theorem 4.7]{HH1} that $S(q\rho^*(D-\Delta)) z \subseteq ((\mathbf{x}^{qs})S)^*$. 
Let $c \in S$ be a big test element for $S$ and set $I_{e}= S(q\rho^*(D-\Delta))M_e \subset S$ for each $q=p^e$.  
Then $c I_{e} z \subseteq  (\mathbf{x}^{qs}) M_e$ by the definition of big test elements, which implies that 
$c I_{e}\mathcal{K}(\mathbf{x}^{qs}, \infty, M_e)=0$ for all $s \ge 1$ and all $q=p^e$. 

Now applying \cite[Lemma A.3]{Ha} (cf.~\cite{Mc}) to the exact sequence 
\[
0 \to J^{[q]}S(q\rho^*\Delta) \to M_e \to M_e/J^{[q]}S(q\rho^*\Delta) \to 0, 
\]
we see that 
\[
c I_{e} \mathcal{K}\bigl(\mathbf{x}^{qs}, \infty, J^{[q]}S(q\rho^*\Delta)\bigr) \subset \mathcal{K}\bigl(\mathbf{x}^{qs}, 2, J^{[q]}S(q\rho^*\Delta)\bigr) 
\]
for all $s \ge 1$ and all $q=p^e$. 
Choose an element $d_e \in I_e$ such that 
\[\nu(d_e)=\min \{\nu(f) \, | \, f \in I_e\}=\nu(I_e)\] for every $e \ge 0$, and we will check that the conditions (i)-(iii) in Lemma \ref{key prop} are satisfied for the sequence $\{cd_e\}$. 
We have already seen that (iii) is satisfied for $\{c d_e\}$. 
Since $\lim_{e \to \infty}\nu(cd_e)/p^e=\lim_{e \to \infty} (\nu(c)+\nu(I_e))/p^e=0$ by $(\spadesuit)$, the condition (ii) is also satisfied for $\{c d_e\}$. 
Finally we use the containment $I_e^p \subset I_{e+1}$ for every $e \ge 0$ to verify the condition (i) for $\{cd_e\}$. 
Thus, the assertion follows from Lemma \ref{key prop}. 
\end{proof}

\begin{rem}
(1) When $\Delta=0$ and $\ba=R$, by applying \cite[Proposition 8.23 (e)]{HH1} (resp. \cite[Theorem 3.6 (2)]{AE}, \cite[Theorem 3.1]{HH}) instead of Lemma \ref{finite test completion}  (resp. Remark \ref{finite test rem} (2), Proposition \ref{valuative criterion}), one can show that Theorem \ref{main thm} holds without assuming that $R$ is $F$-finite.  

(2) 
As a generalization of $\Q$-Gorensteinness, the finite generation of the anti-canonical ring is discussed in \cite{CEMS}. 
In the setting of Theorem \ref{main thm}, assume that the anti-canonical ring $\bigoplus_{i \ge 0}\sO_X(-\lfloor i(K_X+\Delta) \rfloor)$ is finitely generated instead of assuming that $K_X+\Delta$ is numerically $\Q$-Cartier.  
It then follows from \cite[Corollary 5.7]{CEMS} that $\tau_{\rm b}(R, \Delta)=\tau_{\rm fg}(R, \Delta)$.
We do not know whether the equality $\tau_{\rm b}(R, \Delta, \ba^t)=\tau_{\rm fg}(R, \Delta, \ba^t)$ holds for any $\ba \subset R$ and $t>0$ under the assumption that $\bigoplus_{i \ge 0}\sO_X(-\lfloor i(K_X+\Delta)\rfloor)$ is finitely generated. 
\end{rem}

Since every normal surface is numerically $\Q$-factorial, the following holds. 
\begin{cor}\label{two-dimensional}
Let $(R, \Delta, \ba^t)$ be a triple where $R$ is a two-dimensional $F$-finite normal domain of characteristic $p>0$. Then 
$\tau_{\rm b}(R, \Delta, \ba^t)=\tau_{\rm fg}(R, \Delta, \ba^t)$. 
\end{cor}

\begin{proof}
First, $\Spec R$ is numerically $\Q$-factorial by Example \ref{eg num Cartier} (2).  
Second, using the Hodge index theorem \cite[Theorem 10.1]{Ko}, we can verify that Lemma \ref{valuation} holds for all two-dimensional excellent normal integral schemes. 
Thus, the assertion follows from an argument very similar to the proof of Theorem \ref{main thm}. \end{proof}

\begin{rem}
(1) When $\Delta=0$ and $\ba=R$, Corollary \ref{two-dimensional} follows from \cite[Theorem 8.8]{LS}. 

(2) Combining Corollary \ref{two-dimensional} and \cite[Theorem 1]{dFDTT}, we see that if $X=\Spec R$ is a normal affine surface over an uncountable algebraically closed field of characteristic zero and $\Delta$ is an effective $\Q$-Cartier $\Q$-Weil divisor on $X$, then the multiplier ideal $\mathcal{J}(X, \Delta)$ in the sense of \cite{dFH} coincides, after reduction to characteristic $p \gg 0$, with the finitistic test ideal $\tau_{\rm fg}(R_p, \Delta_p)$ associated to modulo $p$ reduction $(R_p, \Delta_p)$ of $(R, \Delta)$. 
\end{rem}


\end{document}